\documentclass[12pt]{amsart}
\usepackage{graphicx}
\usepackage[english]{babel}

\usepackage[utf8]{inputenc}
\usepackage{amsmath,amsfonts,amssymb,amscd,amsthm,txfonts,hyperref}
\usepackage[all]{xy}
\usepackage{multibib}

\usepackage{newunicodechar}

\newunicodechar{≥}{\ensuremath{\geq}}

\newtheorem{theorem}{Theorem}[section]
\newtheorem{proposition}{Proposition}[section]
\newtheorem{lemma}{Lemma}[section]

\newtheorem{remark}{Remark}[section]

\theoremstyle{definition}

\newcommand{\R}{\varmathbb R}

\newcommand{\N}{\varmathbb N}

\newcommand{\E}{\varmathbb E}

\title[Scattering for one dimensional Hartree Fock]{Scattering for the one dimensional Hartree Fock equation}
\author{Cyril Malézé}
\address{Centre Mathématique Laurent Schwartz\\
\'Ecole Polytechnique, CNRS, Université Paris-Saclay\\
Palaiseau, 91128 Cedex, France}
\email{cyril.maleze@polytechnique.edu}

\begin{document}
\maketitle

\begin{abstract}
    We consider the Hartree-Fock equation in 1D, for a small and localised initial data and a finite measure potential. We show that there is no long range scattering due to a nonlinear cancellation between the direct term and the exchange term for plane waves. We employ the framework of space-time resonances that enables us to single out precisely this cancellation and to obtain scattering to linear waves as a consequence. 
\end{abstract}

\section{Introduction}

\subsection{The Hartree-Fock equation and the framework of random fields}\label{subsec:intro} We consider the following time-dependent Hartree-Fock equation for fermions
\begin{equation} \label{Cauchyprob-density-matrices}
i\partial_t \gamma=[-\Delta+w*\rho+\mathcal X, \gamma]
\end{equation}
for a nonnegative self-adjoint operator $\gamma$ on $L^2(\mathbb R^d)$ with kernel $k$, density $\rho(x)=k(x,x)$ and where $\mathcal X$ denotes the exchange term operator with kernel $-w(x-y)k(x,y)$, where $w$ is an even pairwise finite measure and represents the interaction potential.

Equation \eqref{Cauchyprob-density-matrices} is the standard Hartree-Fock equation for density matrices, and can be reformulated in the probability framework:\begin{equation}\label{Cauchyprob}
 i\partial_t X = -\Delta X + (w*\E [|X|^2]) X - \int_{\R^d}  w(x-y)\E[ \overline{ X (y)} X(x)] X(y) \, dy ,
\end{equation}
where $X:[0,T]\times \mathbb R^d \times \Omega\mapsto \mathbb C$ is a random field defined over a probability space $(\Omega, \mathcal A,\mathbb P )$, $\E$ denotes the expectancty on $\Omega$. The introduction of the probability framework for the Hartree and the Hartree-Fock equation was already presented and motivated for instance in \cite{collot2023stability,CdS,CdS2,malézé2023scattering}.

The time-dependent Hartree-Fock equation is a mean-field equation for the dynamics of large Fermi systems. The derivation of \eqref{Cauchyprob-density-matrices} in the mean field limit was first done in \cite{BardosDerivation}, and was extended to the case of unbounded interaction potentials, as the Coulomb one, in \cite{FrohlichDerivationFermi}. Estimates for the convergence in the semi-classical limit that arises for confined Fermi systems were proved in \cite{benedikter2014mean,EESY} , and were extended recently to mixed states, and conditionally to more singular interaction potentials, in \cite{Benedikter2016Dec,porta2017mean}. Another derivation by different techniques, and including another large volume regime for long-range potentials, was given in \cite{BBPPT,Petrat2016Mar}. The Cauchy problem for localised solutions to \eqref{Cauchyprob-density-matrices} was studied in \cite{BDF,BDF2,chadam1976time,CG,Z}. One can check \cite{collot2023stability} for a quick formal derivation of \eqref{Cauchyprob}.

Equations \eqref{Cauchyprob-density-matrices} and \eqref{Cauchyprob} have received less attention than the static Hartree-Fock equation or the reduced Hartree-Fock equation 
\begin{equation}\label{Cauchyprob-reduced}
i\partial_t X = -\Delta X + (w*\E [|X|^2]) X ,
\end{equation}
in the probability framework, and
 \begin{equation} \label{Cauchyprob-density-matrices-reduced}
i\partial_t \gamma=[-\Delta+w*\rho, \gamma],
\end{equation}for density matrices. We mention that the exchange term in the Hartree equation does not appear in the case of bosonic systems, but always appear for fermions. It is due to the form of the canonical wave functions considered for each system. The exchange term for fermionic system is often negligible compared to the direct term (see \cite{Petrat2016Mar} for example), but in some cases it can be relevant to study the influence of the exchange term on the dynamics of the system.

Most of the results describing the dynamics have so far been obtained for the equations without exchange term \eqref{Cauchyprob-reduced} and \eqref{Cauchyprob-density-matrices-reduced}. Indeed, the contribution of the exchange term is negligible in certain regimes. This is the case in the semi-classical limit to the Vlasov equation, which was studied in \cite{GIMS98} for a system with a finite number of particles, and in \cite{Benedikter_2016}, \cite{benedikter2014mean} in the limit of number of particles going to infinity. However, we observe in this paper some cancellation phenomena when considering the exchange terms, see Subsection \ref{subsec:cancellation}. 

The exchange term can also be approximated as a function of the density $\mathbb E (|X|^2)=\rho$, as in Density Functional Theory (see e.g. \cite{CF} for a review and \cite{J,SCB} for the Cauchy problem of time-dependent Kohn-Sham equations).

The reduced Hartree equations \eqref{Cauchyprob-reduced} and \eqref{Cauchyprob-density-matrices-reduced} admit nonlocalised equilibria that models a space-homogeneous electron gas. The stability of such equilibria has been studied in \cite{CHP,CHP2,LS2,LS1} for Equation \eqref{Cauchyprob-density-matrices-reduced} and \cite{CdS,CdS2,H} for Equation \eqref{Cauchyprob-reduced}. The stability of the zero solution for the time-dependent Kohn-Sham equation was showed in \cite{PS}. In \cite{NAKASTUD}, Hadama proved the stability of steady states for the reduced Hartree equation in a wide class, which includes Fermi gas at zero temperature in dimension greater than $3$, with smallness asumption on the potential function. In a recent work \cite{NY23}, Nguyen and You proved that the symbol of the linearised problem could not be inverted in the case of the Coulomb interaction potential. However, they were still able to describe and to prove some time decay for the linearised dynamics. The Hartree-Fock equation \eqref{Cauchyprob} also admits nonlocalised equilibria studied in \cite{collot2023stability}.

\subsection{Long range scattering for NLS and Hartree-Fock equations}\label{longrangescatt}
Long range scattering is a way to describe the long-time dynamics of a solution of a dispersive equation. In some cases, the solution has a different asymptotic behavior than the one of a linear solution of the underlying equation, hence the solution does not enjoy scattering, and we need to incorporate a logarithmic phase correction to describe the asymptotic. In this case, we say that the solution enjoy long range scattering. For equation \eqref{Cauchyprob-density-matrices}, we do not expect scattering in dimension 1 in $L^2$. As explained in \cite{livreTao} (principle 3.1), the critical regularity for scattering for a solution of \begin{equation}\label{eq:NLS}
    \left \{ 
    \begin{array}{rcl}
         i\partial_t u +\Delta u & = & \pm\lvert u \rvert^{p-1} u\\
         u(0) & = & u_0
    \end{array} \right . ,
\end{equation}   
is $s_c=\frac{d}{2}-\frac{2}{p-1}$, where $d$ is the dimension and $p$ the order of the non-linearity. In the case of \eqref{Cauchyprob-density-matrices}, with $w$ a finite measure, we expect a behaviour similar to the behaviour of the solution of \eqref{eq:NLS} with $d=1$ and $p=2$, that does not scatter in $L^2$.  

The long range scattering for NLS has been widely investigated. Many methods have been introduced, we mention some of them. One can found argument relying on a bootstrap estimate involving
the energy-type norm in the work \cite{HayashiNaumkin1998}. We also mention the work \cite{Lindblad2005Dec}, where the author introduce a new variable similar to the profile of the solution and then using energy estimates on this new variable. Another method used in \cite{Ifrim2015Jun}, consists in a wave packet decomposition, using a factorization of the Schrödinger linear group as seen in \cite{Ginibre1993Jan}. Another approach in \cite{Deift2011Jun} capitalizes on the complete integrability of the 1d cubic NLS through the use of inverse scattering. In \cite{KatoPusateri2011}, the authors use the theory of space time resonances to  show long range scattering for small initial data, for the non-linear Schrödinger (NLS) equation in dimension 1, and for the reduced Hartree-Fock equation with a Coulomb potential in dimension greater than 2. 
This two cases are critical from the point of view of long time asymptotic behavior. For details on long range scattering for NLS and on the aforementioned papers, one can check the review \cite{murphy2021review}. 

Like the previous work of Kato and Pusateri \cite{KatoPusateri2011}, the main ideas in this paper come from the theory of space time resonances, and are mostly inspired by the works of Germain, Masmoudi and Shatah \cite{germain2009global,germain2012global2,germain2012global}. This theory was used for instance to study scattering for the Gross-Pitaevskii equation in three dimensions \cite{gustafson2009scattering}, or to show global asymptotic stability of solitary waves of the nonlinear Schrödinger equation in space dimension 1 \cite{collot2023asymptotic}.

In \cite{KatoPusateri2011}, the authors introduced a new functional framework that allows a stationary phase argument, that we extend in this paper to the case of the Hartree-Fock equation in dimension 1, for potentials that are finite measures. Using the point of view of \cite{KatoPusateri2011}, allows us to make appear some cancellation for the Hartree-Fock equation that suggests that the solution of \eqref{Cauchyprob} scatters, which contradicts our first expectation. The results in \cite{KatoPusateri2011} are similar to the ones in \cite{HayashiNaumkin1998}, but the approach is different and more insightful as it can be extended to other cases as the one treated here.

Concerning long-range scattering for the Hartree-Fock equation, we can mention the works of Wada \cite{Wada2002Jan} and Ikeda \cite{Ikeda2012Jan}, where the authors obtain long-range scattering for the Hartree-Fock equation with a potential of the form $w(x)=\lambda\lvert x\rvert^{-1}$. The method used in those works are derived from \cite{Ginibre1993Jan}.

\subsection{Cancellation for plane waves solutions}\label{subsec:cancellation}

The results we mentioned in subsection \ref{longrangescatt} on long time dynamics for Hartree equations and NLS in dimension 1, suggest that we could expect long range scattering for solutions of \eqref{Cauchyprob} with $d=1$, for small and localised initial data. However, we claim that in fact we have the scattering result \eqref{ineq:scattering}. To explain this phenomenon, one can notice that plane wave functions of the form:

\begin{equation}\label{eq:planewavesol}
    X(t,x,\omega)=f(\omega)e^{ix\xi-it\xi^2},
\end{equation}for any function $f\in L^2(\Omega)$ and $\xi\in \R$ are solutions of $\eqref{Cauchyprob}$ for which the non-linearity cancels out. This is because the direct and exchange terms cancel one another for such solutions. Then, because a small and localised solution of \eqref{Cauchyprob} is expected for large time to undergo either long range scattering or standard scattering, it should be close to a superposition of localised plane wave solutions of the form \eqref{eq:planewavesol}, and the non-linearity in \eqref{Cauchyprob} would vanish to leading order because of the cancellation. Thus, the solution would be in this case can be approximated by a solution of the linear part of \eqref{Cauchyprob}, and the solution should enjoy scattering. 

We mention that when studying nonlocalised equilibria of \eqref{Cauchyprob} (see \cite{collot2023stability}), the authors did not use this cancellation to prove the scattering results, as they studied non-localised equilibria, and they might have not expected a cancellation of this type. Moreover, in \cite{Ikeda2012Jan,Wada2002Jan}, results of long range scattering are obtained in dimension greater than 2 for the Coulomb potential for the Hartree-Fock equation. In this case, the interactions of plane wave functions in the solutions are non-localised and thus the cancellation that we highlight in the present framework does not appear. 

\subsection{Main result}

In this paper, we prove the following, recalling that $\mathfrak S^1$ is the standard Schatten space 1-norm for operators on $L^2(\R)$
\begin{theorem}\label{Cor:scatteringoper}
    Let $w$ be a finite measure and $\gamma_0$ an operator on $L^2(\R)$, such that
    $$\text{Tr}(\langle \nabla \rangle \gamma_0)+\text{Tr}(\langle x \rangle \gamma_0)\leq \varepsilon,$$

    for $\varepsilon>0 $ small enough (depending on $w$), there exists a unique global solution $\gamma\in L_t^\infty,\mathfrak S^1$ of \eqref{Cauchyprob-density-matrices} with initial condition $\gamma(t=0)=\gamma_0$, that scatters at infinity, in the sense that there exists $\gamma_\infty\in \mathfrak S^1$ such that for $t\geq 1$: 
    \begin{equation}\label{ineq:scatteringopdens}
       \| \gamma(t)-e^{it\Delta}\gamma_\infty\|_{\mathfrak S^1}\leq t^{-\delta}.
    \end{equation}
\end{theorem}

\begin{remark}
    A unique local solution is provided by \cite{BDF,BDF2}, we prove in this paper that it is global and scatters at infinity in $\mathfrak S^1$. 
\end{remark}

\begin{remark}
    The equation is time reversible, so the same results for negative times holds.
\end{remark}

We prove Theorem \ref{Cor:scatteringoper} as a consequence of the following scattering theorem in the probality framework (for the notations of the functional spaces and their related norms, see subsection \ref{subsec:Notations})

\begin{theorem}\label{th:principal}

Let $w$ be a finite measure and $\theta\in [0,1)$. There exists $\varepsilon>0$ such that for some $0<\delta<\frac{1}{4}$ and for any $X_0\in H^{1,0},L_\omega^2\cap H^{0,1},L_\omega^2$, with $\|X_0\|_{H^{1,0},L_\omega^2\cap H^{0,1},L_\omega^2}\leq \varepsilon$, there exists a unique global solution $X\in \mathcal C(\R,H^{1,0},L_\omega^2\cap H^{0,1},L_\omega^2)$ of \eqref{Cauchyprob} with initial condition $X(t=1)=e^{i\Delta}X_0$, which scatters at infinity, in the sense that there exists $\Hat W\in H^{\theta,0}_\xi,L_\omega^2\cap H^{0,\theta}_\xi,L_\omega^2\cap L_x^\infty,L_\omega^2$ such that for $t\geq 1$ 

\begin{equation}\label{ineq:scattering}
    \|\hat{X}(t)-e^{-it\xi^2}\Hat W\|_{H^{\theta,0}_\xi\cap H^{0,\theta}_\xi,L_\omega^2\cap L_\xi^\infty,L_\omega^2}\leq t^{-\delta}.
\end{equation}

\end{theorem}

\begin{remark}
    The connections between the two equations is explained in subsection \ref{subsec:intro}, and is made rigorous in \cite{ASdesuzzonglobwp} and section \ref{Sec:proofthop}.
\end{remark}

\begin{remark}
    We almost reach scattering in $H^{1,0}\cap H^{0,1}$ which is the set for the initial data. This slightly improves results in \cite{KatoPusateri2011}, and we mention that in \cite{Wada2002Jan}, the author obtains a similar almost optimal result.
\end{remark}

\subsection{Notations}\label{subsec:Notations}

Our notation for the Fourier transform is
$$
\hat f(\xi) =\mathcal F f (\xi)=(2\pi)^{-\frac 12}\int_{\mathbb R} e^{-ix\xi}f(x)\, dx.
$$

For $m,l\in \R$ we denote $L_t^pH_x^{m,l}L_\omega^2$ the space $$\langle x\rangle^l\langle\nabla\rangle^m L^p(\R,L^2(\R^d,L^2(\Omega))),$$
with the norm: $$\lvert\lvert u\rvert\rvert_{L_t^pH_x^{m,l}L_\omega^2}=\lvert \lvert \langle x\rangle^l\langle\nabla\rangle^mu\rvert\rvert_{L_t^pL_x^{2}L_\omega^2},$$
where $\langle x\rangle:=(1+\lvert x\rvert^2)^\frac{1}{2}$, $\langle \nabla\rangle:=\mathcal{F}^{-1}\langle \xi\rangle\mathcal F$.

We also denote $L_t^p\Dot H_x^{m,l}L_\omega^2$ the space $$\lvert x\rvert^l\lvert\nabla\rvert^m L^p(\R,L^2(\R^d,L^2(\Omega))),$$
with the norm: $$\lvert\lvert u\rvert\rvert_{L_t^p\Dot H_x^{m,l}L_\omega^2}=\lvert \lvert\ \lvert x\rvert^l\lvert\nabla\rvert^m u\rvert\rvert_{L_t^pL_x^{2}L_\omega^2},$$
where $\lvert \nabla\rvert:=\mathcal{F}^{-1}\lvert \xi\rvert\mathcal F$.

We also denote by $M^1(\R)$ the set of finite measures on $\R$, with the associated norm $\|\cdot\|_{M^1}$.

We denotre $A(t,x,\omega)\lesssim B(t,x,\omega)$ if there is a constant $C>0$ such that $$A(t,x,\omega)\leq C\ B(t,x,\omega).$$

We denote by $|f\rangle\langle g|$ the operator $$|f\rangle\langle g|(v)(x)=\int \overline{g(y)}v(y)dy\ f(x).$$
 
\subsection{Organization of the paper}

The paper is organized as follows. In Section \ref{sec:ideas} we introduce the local solution of \eqref{Cauchyprob} for small initial data. In Section \ref{Section:proofofprop}, we prove an estimate for the local solution and in Section \ref{Section:proofofth} we use this estimate and its proof to show that the solution is global for a small enough initial data and the scattering result \eqref{ineq:scattering}. In the last section, we prove Theorem \ref{Cor:scatteringoper}.

\subsection{Acknowledgments}

This result is part of the ERC starting grant project FloWAS that has received funding from the European Research Council (ERC) under the Horizon Europe research and innovation program (Grant agreement No. 101117820). The author would like to thank Prof. Anne-Sophie de Suzzoni (supported by the S. S. Chern Young Faculty Award funded by AX) and Prof. Charles Collot for their guidance and support. They both gave many most helpful comments during the writing of this paper.

\section{Local solution of \eqref{Cauchyprob}}\label{sec:ideas}

In this section we introduce the local solution of \eqref{Cauchyprob} and announces some results.

We define, in analogy with \cite{KatoPusateri2011}, for $0<\alpha<min(\frac{1}{4},\frac{1-\theta}{4(3-2\theta)})$, the norm $\|\cdot\|_{\mathcal{X}_T}$ 

\begin{equation}
    \|X\|_{\mathcal{X}_T}:=\|t^\frac{1}{2}X\|_{L_T^\infty,L_x^\infty,L_\omega^2}+\|t^{-\alpha}X\|_{L_T^\infty,\Dot{H}_x^{1,0},L_\omega^2}+\|t^{-\alpha}Z\|_{L_T^\infty,\Dot{H}_x^{0,1},L_\omega^2}+\| X\|_{L_T^\infty,L_x^2,L_\omega^2}  
\end{equation}

and the space 

\begin{equation}\label{def:X_T}
    \mathcal{X}_T:=\{ X\in L_T^\infty,L_x^\infty,L_\omega^2 \cap \mathcal{C}(\R, H_x^{1,0},L_\omega^2\cap H_x^{0,1},L_\omega^2),\ \|X\|_{\mathcal{X}_T}<\infty\}. 
\end{equation}

The scattering solution will be constructed as a continuation of the local solution provided by the following result 

\begin{theorem}\label{th:local-existence}
     Let $T>1$. There exist $C_0,\ \varepsilon>0$ such that for any $X_0\in H^{1,0},L_\omega^2\cap H^{0,1},L_\omega^2$, with $\|X_0\|_{H^{1,0},L_\omega^2\cap H^{0,1},L_\omega^2}\leq \varepsilon$, there exists a unique solution $X\in \mathcal C([0,T],H^{1,0},L_\omega^2\cap H^{0,1},L_\omega^2)$ of \eqref{Cauchyprob} such that: $\|X\|_{\mathcal{X}_T}\leq C_0\| X_0\|_{H^{1,0},L_\omega^2\cap H^{0,1},L_\omega^2}$.    
\end{theorem}

\begin{remark}
    To prove this theorem, we can apply classical methods presented for instance in \cite{cazenave1990cauchy,ginibre1979class,ginibre1980class}, that is to say, by solving a fixed point problem. One can also check the sketch of proof, presented in Appendix \ref{appendix:locex}. 
\end{remark}

We first prove in Section \ref{Section:proofofprop} the following proposition: 

\begin{proposition}\label{prop:bootstrap-estimate}
    The solution $X$ given by Theorem \ref{th:local-existence} satisfies: \begin{equation}\label{ineq:boostrap-estimate}
        \|X\|_{\mathcal{X}_T}\leq \varepsilon +C_1\|X\|_{\mathcal{X}_T}^3,
    \end{equation} for some constant $C_1>0$ independent of $T$.
\end{proposition}

In Section \ref{Section:proofofth}, we use Proposition \ref{prop:bootstrap-estimate} and a boostrap argument to prove that the solution is global, before showing the scattering result \eqref{ineq:scattering}.  

Before turning to the proofs, we recall the following estimate of the Schrödinger linear semigroup \begin{lemma}(\cite{HayashiNaumkin1998})\label{lemma:sch}
    For any $\gamma>\frac{n}{2}+2\beta$: $$\|e^{it\Delta}g\|_{L^\infty(\R^n)}\lesssim \frac{1}{t^{n/2}}\|\hat{g}\|_{L^\infty(\R^n)}+\frac{1}{t^{\frac{n}{2}+2\beta}}\|g\|_{H^{0,\gamma}(\R^n)}.$$

\end{lemma}

\section{Proof of Proposition \ref{prop:bootstrap-estimate}}\label{Section:proofofprop}

\subsection{Setting the stationary phase argument}\label{subsec:computationofC}

In this subsection, we set a stationary phase argument. 

Writing Duhamel's formula for the Cauchy problem \eqref{Cauchyprob} gives 

\begin{equation}\label{eq:DuhamelformX}
    X(t)=e^{i(t-1)\Delta}X_0+e^{it\Delta}C(X,X,X),
\end{equation}

with $$C(X,X,X)=-i\int_1^te^{-is\Delta}\Big( \big(w*\E[\lvert X(s)\rvert^2]\big)X(s)-\int w(x-y)\E[\overline{X(s,y)}X(s,x)]X(s,y)dy \Big)ds.$$

We set $Z=e^{-it\Delta}X$ the profile of the solution, then it verifies

\begin{equation}\label{zzzz}
    \hat{Z}(t,\xi)=\hat{X_0}(\xi)+\hat{C}(X,X,X)(\xi).
\end{equation}

To understand the dynamic in long time, we compute $\hat{C}$.

\begin{proposition}
    The function $C(X,X,X)$ verifies 

    \begin{equation}\label{reloueq}
        \hat{C}(X,X,X)=I_1+I_2,
    \end{equation}

    with

    $$I_1=-i\mathcal{F}_x\Bigg( \int_1^t ds\ e^{-is\Delta} (w*\E[\lvert X(s)\rvert^2])X(s))\Bigg), $$

    and 
    $$I_2=i\mathcal{F}_x\Bigg( \int_1^t ds\ e^{-is\Delta}\int dy\ w(x-y)\E[\overline{X(s,y)}X(s,x)]X(s,y)dy \Bigg).$$
\end{proposition}

\begin{proof}
We compute $I_1$ first. Using the Fourier transform of a product

$$I_1  =  \frac{-i}{\sqrt{2\pi}}\int_1^tds\ e^{is\xi^2}\big(\mathcal{F}_x(w*\E[\lvert X(s)\rvert^2]\big)*\mathcal{F}_x(X(s)),$$

and then the Fourier transform of a convolution

$$I_1=-i\int_1^tds\ e^{is\xi^2} \int d\eta\ \hat{w}(\eta)\mathcal{F}_x\big(\E[\lvert X(s)\rvert^2]\big)(\eta)\hat{X}(s,\xi-\eta).$$

Finally, we get

$$I_1=\frac{-i}{\sqrt{2\pi}}\int_1^tds\ e^{is\xi^2} \int d\eta\ \hat{w}(\eta) \E[\int d\sigma \overline{\hat{X}(s,\sigma-\eta)}\hat{X}(s,\sigma)]\hat{X}(s,\xi-\eta).$$

Recalling that: $\hat{X}(s,\xi)=e^{-is\xi^2}\hat{Z}(s,\xi)$, we obtain

$$I_1=\frac{-i}{\sqrt{2\pi}}\int_1^tds\ e^{is\xi^2} \int d\eta d\sigma\ \hat{w}(\eta)\E[e^{is(\eta-\sigma)^2}\overline{\hat{Z}(s,\sigma-\eta)}e^{-is\sigma^2}\hat{Z}(\sigma)]e^{-is(\xi-\eta)^2}\hat{Z}(\xi-\eta).$$

Thus

\begin{equation}\label{eq:I_1}
    I_1=\frac{-i}{\sqrt{2\pi}}\int_1^tds \int d\eta d\sigma\ e^{2is\eta (\xi-\sigma)}\hat{w}(\eta)\E[\overline{\hat{Z}(s,\sigma-\eta)}\hat{Z}(\sigma)]\hat{Z}(\xi-\eta).
\end{equation}

Following the same ideas, we compute $I_2$. 

First we write

$$I_2=i\int_0^t ds\ e^{-is\xi^2}\mathcal{F}_x\bigg( w*\big(\E[\overline{X(s,\cdot)}X(s,x)]X(s,\cdot)\big)\bigg).$$

We compute that

$$\mathcal{F}_x\big( w*(\E[\overline{X(s,\cdot)}X(s,x)]X(s,\cdot))\big)=\int dx\ e^{-i\xi x}\int dy \ w(x-y)\E[\overline{X(s,y)}X(s,x)]X(s,y). $$

Using the formula of the inverse Fourier transform, we have 

$$\mathcal{F}_x\big( w*(\E[\overline{X(s,\cdot)}X(s,x)]X(s,\cdot))\big)  = \frac{1}{\sqrt{2\pi}} \int dydx\ e^{-i\xi x} w(x-y)\E\Big[\overline{X(s,y)}\int d\eta \ e^{i\eta x}\hat{X}(s,\eta)\Big]X(s,y).$$

And then, proceeding as for $I_1$, we obtain

\begin{equation}\label{eq:I_2}
    I_2=\frac{i}{\sqrt{2\pi}}\int_1^tds \int d\eta d\sigma\ e^{2is\eta (\xi-\sigma)}\hat{w}(\eta)\E[\overline{\hat{Z}(s,\sigma-\eta)}\hat{Z}(\xi-\eta)]\hat{Z}(\sigma).
\end{equation}

Combining \eqref{eq:I_1} and \eqref{eq:I_2} gives
\begin{equation}\label{eq:C}
    \hat{C}=\frac{-i}{\sqrt{2\pi}}\int_1^tds \int d\eta d\sigma\ e^{2is\eta (\xi-\sigma)}\hat{w}(\eta)\Bigg(\E[\overline{\hat{Z}(s,\sigma-\eta)}\hat{Z}(\sigma)]\hat{Z}(\xi-\eta)-\E[\overline{\hat{Z}(s,\sigma-\eta)}\hat{Z}(\xi-\eta)]\hat{Z}(\sigma)\Bigg),
\end{equation}
which is the desired result \eqref{reloueq}.
\end{proof}

\begin{proposition}
    $Z$ is a solution of 

    $$\hat{Z}(t,\xi)=\Hat{X}_*(\xi)+\int_1^s R(s,\xi)ds,$$

    where

    \begin{equation}\label{eq:defR}
    R(s,\xi)=\frac{-i}{\sqrt{2\pi}}\int dxdy\ \frac{1}{2s}(e^{\frac{ixy}{2s}}-1)\mathcal{F}^{-1}_{\eta,\sigma}(F)(s,x,y,\xi),
\end{equation}

and $$F(s,\eta,\sigma,\xi)=\hat{w}(\eta)\Big(\E[\overline{\hat{Z}(s,\xi-\sigma-\eta)}\hat{Z}(\xi-\sigma)]\hat{Z}(\xi-\eta)-\E[\overline{\hat{Z}(s,\xi-\sigma-\eta)}\hat{Z}(\xi-\eta)]\hat{Z}(\xi-\sigma)\Big).$$
\end{proposition}

\begin{proof}

Using a change of variable and using Plancherel's formula we get, using \eqref{eq:C}

\begin{equation}
    \hat{C}(X,X,X)=\frac{-i}{\sqrt{2\pi}}\int_1^tds \int dxdy\ \mathcal{F}_{\eta,\sigma}\big(e^{2is\eta \sigma}\big)(x,y)\mathcal{F}^{-1}_{\eta,\sigma}(F)(s,x,y,\xi),
\end{equation}

And thus $$ \hat{C}(X,X,X)=\frac{-i}{\sqrt{2\pi}}\int_1^tds \int dxdy\ \frac{1}{2s}e^{\frac{ixy}{2s}}\mathcal{F}^{-1}_{\eta,\sigma}(F)(s,x,y,\xi).$$

Noticing that $F(s,0,0,\xi)=0$ we get

$$\int dxdy\ \mathcal{F}^{-1}_{\eta,\sigma}(F)(s,x,y,\xi)=0,$$

and then we can write

$$\hat{C}(X,X,X)=\frac{-i}{\sqrt{2\pi}}\int_1^tds \int dxdy\ \frac{1}{2s}(e^{\frac{i\eta \sigma}{2s}}-1)\mathcal{F}^{-1}_{\eta,\sigma}(F)(s,x,y,\xi).$$

Finally, we obtain

\begin{equation}\label{eq:duhamelZ}
    \hat{Z}=\hat{X}_*+\int_1^t R(s,\xi)ds,
\end{equation}

where

\begin{equation}
    R(s,\xi)=\frac{-i}{\sqrt{2\pi}}\int dxdy\ \frac{1}{2s}(e^{\frac{i\eta \sigma}{2s}}-1)\mathcal{F}^{-1}_{\eta,\sigma}(F(s,\eta,\sigma,\xi))(x,y),
\end{equation}

    which is the desired identity \eqref{eq:defR}.
\end{proof}

\subsection{Estimate on $R$}

Setting a stationary phase argument to the oscillatory integral with respect to the variables $\eta$ and $\sigma$, as done in \cite{KatoPusateri2011}, we showed that $\hat{Z}$ is of the form 

$$\hat{Z}(t,\xi)=\Hat{X}_*(\xi)+\int_1^s R(s,\xi)ds.$$

We now claim that $R$ decays faster than $s^{-1}$, and this will imply the scattering result \eqref{ineq:scattering}. In the following, we prove this fact. We begin by proving the estimates on $R$.

The following lemma gives an estimate for $R$: 

\begin{lemma}\label{lemma:estimateR}
    For any $X\in \mathcal{X}_T$, for $s\geq 1$ we have the following estimate

    \begin{equation}\label{eq:estimateforR}
        \| R(s,\xi)\|_{L_\omega^2}\lesssim s^{-1-\delta+3\alpha}\|X\|_{\mathcal{X}_T}^3,
    \end{equation}
    for any $0<\delta<\frac{1}{4}$.
\end{lemma}

\begin{proof}
    From the expression \eqref{eq:defR}, we immediately get that for any $0<\delta<\frac{1}{2}$ 

    \begin{equation}\label{rasleboldeslabel}
        \rvert R(s,\xi)\lvert\lesssim s^{-1-\delta}\int dxdy \lvert x \rvert^\delta\lvert y \rvert^\delta\lvert\mathcal{F}^{-1}_{\eta,\sigma}(F)(s,x,y,\xi)\rvert.
    \end{equation}

    Next we give an estimate of $\mathcal{F}^{-1}_{\eta,\sigma}(F)(s,\eta,\sigma,\xi)$. First we write

    $$F(s,\eta,\sigma,\xi)=\hat{w}(\eta)(F_1(s,\eta,\sigma,\xi)-F_2(s,\eta,\sigma,\xi)),$$

    with $$F_1(s,\eta,\sigma,\xi)=\E[\overline{\hat{Z}(s,\xi-\sigma-\eta)}\hat{Z}(\xi-\sigma)]\hat{Z}(\xi-\eta),$$
    and $$F_2(s,\eta,\sigma,\xi)=\E[\overline{\hat{Z}(s,\xi-\sigma-\eta)}\hat{Z}(\xi-\eta)]\hat{Z}(\xi-\sigma).$$

    We begin with the $F_1$ part of the Fourier transform

    $$\mathcal{F}^{-1}_{\eta,\sigma}(F_1)=\mathcal{F}^{-1}_{\eta}\Big(\E[\mathcal{F}^{-1}_{\sigma}\Big(\overline{\hat{Z}(s,\xi-\sigma-\eta)}\hat{Z}(\xi-\sigma)\Big)]\hat{Z}(\xi-\eta)\Big).$$

    We compute
    $$\mathcal{F}^{-1}_{\sigma}\Big(\overline{\hat{Z}(s,\xi-\sigma-\eta)}\hat{Z}(\xi-\sigma)\Big)=\sqrt{2\pi}\ \mathcal{F}_\sigma^{-1}(\overline{\hat{Z}(s,\xi-\sigma-\eta)})*\mathcal{F}_\sigma^{-1}(\hat{Z}(\xi-\sigma)).$$

    And by composition of the Fourier transform with a translation we have

    $$\mathcal{F}_\sigma^{-1}(\overline{\hat{Z}(s,\xi-\sigma-\eta)})(x)=e^{ix(\xi-\eta)}\overline{Z(x)},$$
    
    and $$\mathcal{F}_\sigma^{-1}(\hat{Z}(\xi-\sigma))(x)=e^{ix \xi}Z(-x).$$

    So \begin{equation}\label{eq:calculF_1-1}
        \mathcal{F}^{-1}_{\sigma}\Big(\overline{\hat{Z}(s,\xi-\sigma-\eta)}\hat{Z}(\xi-\sigma)\Big)(x)=\sqrt{2\pi}e^{i\xi x}\int dx'\ e^{-ix'\eta}Z(x'-x)\overline{Z}(x')
    \end{equation}

    Then $$\mathcal{F}^{-1}_{\eta,\sigma}(F_1)=\sqrt{2\pi} \mathcal{F}^{-1}_{\eta}\Big(\E[ e^{i\xi x}\int dx'\ e^{-ix'\eta}Z(x'-x)\overline{Z}(x') ]\hat{Z}(\xi-\eta)\Big).$$

    And because $$\mathcal{F}^{-1}_{\eta}(e^{-ix'\eta}\hat{Z}(\xi-\eta))(y)=e^{i\xi(y+x')}Z(x'-y),$$

    we obtain 

    \begin{equation}\label{eq:FourF1}
        \mathcal{F}^{-1}_{\eta,\sigma}(F_1)=\sqrt{2\pi} e^{i\xi(x+y)}\int dx'\ e^{-ix'\xi}\E[Z(x'-x)\overline{Z}(x')]Z(x'-y).
    \end{equation}

    Noticing that the expression of $F_1$ and $F_2$ are symmetrical in $\eta,\sigma$ then we can immediately write

    \begin{equation}\label{eq:FourF2}
        \mathcal{F}^{-1}_{\eta,\sigma}(F_2)=\sqrt{2\pi} e^{i\xi(x+y)}\int dy'\ e^{-iy'\xi}\E[Z(y'-y)\overline{Z}(y')]Z(y'-x).
    \end{equation}

    By Minkowski's inequality and Holder's inequality, we get

    \begin{equation}\label{ineq:FF1}
        \|\mathcal{F}^{-1}_{\eta,\sigma}(F_1)(s,\eta,\sigma,\xi)\|_{L_\omega^2}\lesssim \int dx' \|Z(x'-x)\|_{L_\omega^2}\|Z(x')\|_{L_\omega^2}\|Z(x'-y)\|_{L_\omega^2}, 
    \end{equation}

    and
    \begin{equation}\label{ineq:FF2}
        \|\mathcal{F}^{-1}_{\eta,\sigma}(F_2)(s,\eta,\sigma,\xi)\|_{L_\omega^2}\lesssim \int dy' \|Z(y'-y)\|_{L_\omega^2}\|Z(y')\|_{L_\omega^2}\|Z(y'-x)\|_{L_\omega^2}. 
    \end{equation}

    We use Minkowki's inequality in \eqref{rasleboldeslabel}
    
    $$\|R(s,\xi)\|_{L_\omega^2}\lesssim s^{-1-\delta}\int d\eta d\sigma \lvert \eta \rvert^\delta\lvert \sigma \rvert^\delta\|\mathcal{F}^{-1}_{\eta,\sigma}(F)(s,\eta,\sigma,\xi)\|_{L_\omega^2}.$$

    Thus combining the above inequality with \eqref{ineq:FF1} and \eqref{ineq:FF2}  $$\|R(s,\xi)\|_{L_\omega^2}\lesssim s^{-1-\delta}\int dx dy dz \lvert x \rvert^\delta\lvert y\rvert^\delta\|Z(z-y)\|_{L_\omega^2}\|Z(z)\|_{L_\omega^2}\|Z(z-x)\|_{L_\omega^2}.$$

    Using the fact that $$\lvert x \rvert^\delta\lvert y\rvert^\delta\lesssim \Big( \lvert z-x\rvert^\delta+\lvert z \rvert^\delta\Big)\Big( \lvert z-y \rvert^\delta+\lvert z \rvert^\delta\Big),$$

    we get $$\|R(s,\xi)\|_{L_\omega^2}\lesssim s^{-1-\delta} \|\langle x\rangle^{2\delta}Z\|_{L_x^1,L_\omega^2}^2\|Z\|_{L_x^1,L_\omega^2}. $$

    
    Then by choosing $0<\delta<\frac{1}{4}$: $$\|R(s,\xi)\|_{L_\omega^2}\lesssim s^{-1-\delta}\|Z\|^3_{H_x^{0,1},L_\omega^2}\lesssim s^{-1-\delta-3\alpha}\|X\|_{\mathcal{X}_T}^3, $$

    which is the desired identity \eqref{eq:estimateforR}.
    
\end{proof}

We can now prove Proposition \ref{prop:bootstrap-estimate}.

\begin{proof}[Proof of Proposition \ref{prop:bootstrap-estimate}]
    \textbf{Conservation of the $L^2$ norm}
    
    One can notice that a solution of \eqref{Cauchyprob}, enjoys a conservation of the $L_x^2,L_\omega^2$ norm. Indeed

    $$\frac{d}{d t}\|X\|_{L_{x,\omega}^2}^2=2\text{Im}\langle \int_{\R^d}  w(x-y)\E[ \overline{ X (y)} X(\cdot)] X(y) \, dy,X\rangle_{L_{x,\omega}^2},  $$

    and: $$\langle \int_{\R^d}  w(x-y)\E[ \overline{ X (y)} X(\cdot)] X(y) \, dy,X\rangle_{L_{x,\omega}^2}=\int dxdyd\omega\ w(x-y)\E[ \overline{ X (y)} X(x)] X(y,\omega)\overline{X(x,\omega)}.$$

    Integrating over $\omega$ gives: $$\langle \int_{\R^d}  w(x-y)\E[ \overline{ X (y)} X(\cdot)] X(y) \, dy,X\rangle_{L_{x,\omega}^2}=\int dxdy\ w(x-y)\E[ \overline{ X (y)} X(x)] \E[X(y)\overline{X(x)}]\in \R.$$

    And, thus $\frac{d}{d t}\|X\|_{L_{x,\omega}^2}^2=0$ and the $L_x^2,L_\omega^2$ norm is conserved. 

    \textbf{Estimation of the solution in $\Dot{H}_x^{1,0},L_\omega^2$}

    Recalling the Duhamel's formula \eqref{eq:DuhamelformX}, we give an estimate of the $\Dot{H}_x^{1,0},L_\omega^2$ norm of $C$.

    Using Minkowski's inequality and Young's inequality gives

    $$\Big\|\partial_x \int_1^te^{-is\Delta}\Big( \big(w*\E[\lvert X(s)\rvert^2]\big)X(s)ds\Big\|_{L_{x,\omega}^2}\leq \int_0^t \|w\|_{L_1}\|\partial_x\big(\E[\lvert X(s)\rvert^2]X(s)\big)\|_{L_{x,\omega}^2}ds, $$

    then, by Hölder's inequality \begin{equation*}
        \Big\|\partial_x \int_1^te^{-is\Delta}\Big( \big(w*\E[\lvert X(s)\rvert^2]\big)X(s)ds\Big\|_{L_{x,\omega}^2}\lesssim \int_0^t   \|X(s)\|_{L_x^\infty,L_\omega^2}^2\|\partial_x X(s)\|_{L_{x,\omega}^2}ds,
    \end{equation*}
    and by definition of $\mathcal{X}_T$

    \begin{equation}\label{eq:H011}
        \Big\|\partial_x \int_1^te^{-is\Delta}\Big( \big(w*\E[\lvert X(s)\rvert^2]\big)X(s)ds\Big\|_{L_{x,\omega}^2}\lesssim \int_1^t   s^{\alpha-1}\|X\|_{\mathcal{X}_T}^3 ds\lesssim t^{\alpha}\|X\|_{\mathcal{X}_T}^3.
    \end{equation}

    For the other term in $C$, we write

    $$\int_1^t e^{-is\Delta}\Big(\int w(x-y)\E[\overline{X(s,y)}X(s,x)]X(s,y)dy\Big)ds=\int_1^t e^{-is\Delta}\Big( w*\E[\overline{X(s,\cdot)}X(s,x)]X(s,\cdot)\Big)ds.$$

    And thus, the same steps as above gives 

    \begin{equation}\label{eq:H012}
        \Big\|\partial_x \int_1^t e^{-is\Delta}\Big(\int w(x-y)\E[\overline{X(s,y)}X(s,x)]X(s,y)dy\Big)ds\Big\|_{L_{x,\omega}^2}\lesssim t^{\alpha}\|X\|_{\mathcal{X}_T}^3.
    \end{equation}

    Combining \eqref{eq:H011} and \eqref{eq:H012} gives

    \begin{equation}\label{eq:H01}
        \|\partial_x C\|_{L_{x,\omega}^2}\lesssim t^{\alpha}\|X\|_{\mathcal{X}_T}^3.
    \end{equation}
    
    \textbf{Estimation of the solution in $\Dot{H}_x^{0,1},L_\omega^2$}

    We estimate $\|xC\|_{L_{x,\omega}^2}=\|\partial_\xi \hat{C}\|_{L_{x,\omega}^2}$, recalling expressions \eqref{eq:I_1}, \eqref{eq:I_2} and \eqref{eq:C}, we estimate $\|\partial_\xi I_1\|_{L_{x,\omega}^2}$ and $\|\partial_\xi I_2\|_{L_{x,\omega}^2}$.

    We write that by a change of variable

    $$I_1=\frac{-i}{\sqrt{2\pi}}\int_1^tds \int d\eta d\sigma\ e^{2is\eta\sigma}\hat{w}(\eta)\E[\overline{\hat{Z}(s,\xi-\sigma-\eta)}\hat{Z}(\xi-\sigma)]\hat{Z}(\xi-\eta).$$
    
    Then
    $$\begin{array}{rcl}
        \partial_\xi I_1 & = & \frac{-i}{\sqrt{2\pi}}\Big(\int_1^tds \int d\eta d\sigma\ e^{2is\eta \sigma}\hat{w}(\eta)\E[\overline{\partial_\xi\hat{Z}(s,\xi-\sigma-\eta)}\hat{Z}(\xi-\sigma)]\hat{Z}(\xi-\eta) \\
         & & + \int_1^tds \int d\eta d\sigma\ e^{2is\eta\sigma}\hat{w}(\eta)\E[\overline{\hat{Z}(s,\xi-\sigma-\eta)}\partial_\xi\hat{Z}(\xi-\sigma)]\hat{Z}(\xi-\eta) \\
         & & +\int_1^tds \int d\eta d\sigma\ e^{2is\eta\sigma}\hat{w}(\eta)\E[\overline{\hat{Z}(s,\xi-\sigma-\eta)}\hat{Z}(\xi-\sigma)]\partial_\xi\hat{Z}(\xi-\eta)\Big).
    \end{array}$$

    We only treat the first term of the equality above, the other terms can be treated in a similar way. By a similar computation as in Subsection \ref{subsec:computationofC}, we have 

    $$\begin{array}{r}
         \Big\|\int_1^tds \int d\eta d\sigma\ e^{2is\eta \sigma}\hat{w}(\eta)\E[\overline{\partial_\xi\hat{Z}(s,\xi-\sigma-\eta)}\hat{Z}(\xi-\sigma)]\hat{Z}(\xi-\eta)\Big\|_{L_{x,\omega}^2}   \\
         \lesssim\Big\|\int_1^tds\ e^{-is\Delta}(w*\E[\overline{e^{is\Delta}(xZ)}X])X \Big\|_{L_{x,\omega}^2}.
    \end{array}$$

    Thus, by Minkowski's inequality, young's inequality

    $$\begin{array}{r}
         \Big\|\int_1^tds \int d\eta d\sigma\ e^{2is\eta \sigma}\hat{w}(\eta)\E[\overline{\partial_\xi\hat{Z}(s,\xi-\sigma-\eta)}\hat{Z}(\xi-\sigma)]\hat{Z}(\xi-\eta)\Big\|_{L_{x,\omega}^2}   \\
         \lesssim\int_1^tds \Big\|\E[\overline{e^{is\Delta}(xZ)}X])\|X\|_{L_{\omega}^2}\Big\|_{L_{x}^2},
    \end{array}$$

    and by Hölder's inequality $$\begin{array}{r}
         \Big\|\int_1^tds \int d\eta d\sigma\ e^{2is\eta \sigma}\hat{w}(\eta)\E[\overline{\partial_\xi\hat{Z}(s,\xi-\sigma-\eta)}\hat{Z}(\xi-\sigma)]\hat{Z}(\xi-\eta)\Big\|_{L_{x,\omega}^2}   \\
         \lesssim \int_1^t\|xZ\|_{L_{x,\omega}^2}\|X\|_{L_x^\infty,L_\omega^2}.
    \end{array}$$

    And we get, by adding the other terms of $\partial_\xi I_1$ and using the definition of $\mathcal{X}_T$

    $$\|\partial_\xi I_1\|_{L_{x,\omega}^2}\lesssim \int_1^t\|xZ\|_{L_{x,\omega}^2}\|X\|^2_{L_x^\infty,L_\omega^2}\lesssim \int_1^t   s^{\alpha-1}\|X\|_{\mathcal{X}_T}^3 ds \lesssim  t^{\alpha}\|X\|_{\mathcal{X}_T}^3.$$

    With slight modifications, we can use the same method for the $I_2$ part and we get 

    \begin{equation}\label{eq:H10}
        \|\partial_\xi \hat{C}\|_{L_{x,\omega}^2}\lesssim t^{\alpha}\|X\|_{\mathcal{X}_T}^3.
    \end{equation}
    
    \textbf{Estimation of the solution in $L_x^\infty,L_\omega^2$:}

    Using $\eqref{eq:duhamelZ}$ we have, by Minkowski inequality 

    $$\| \hat{Z}(t,\xi)\|_{L_\omega^2}\leq \| \hat{X}_*(\xi)\|_{L_\omega^2}+\int_1^t \| R(s,\xi) \|_{L_\omega^2}.$$

    And by lemma \ref{lemma:estimateR} one gets

    $$\| \hat{Z}(t,\xi)\|_{L_\omega^2}\lesssim \|X_0\|_{H^{0,1}}+t^{-\delta+3\alpha}\|X\|_{\mathcal{X}_T}^3.$$

    Finally we use lemma \ref{lemma:sch}, with $n=1,3\alpha<\delta$ and $\alpha<\beta<\frac{1}{4}$, to get 

    $$\|X\|_{L^\infty,L_\omega^2}\lesssim \frac{1}{t^{1/2}}\Big(\|X_0\|_{H^{0,1}}+t^{-\delta+3\alpha}\|X\|_{\mathcal{X}_T}^3\Big)+\frac{1}{t^{1/2+\beta}}\|f\|_{H^{0,1},L_\omega^2}.$$

    And we get \begin{equation}\label{eq:Linfty}
        \|X\|_{L^\infty,L_\omega^2}\lesssim \frac{1}{t^{1/2}}\Big(\|X_0\|_{H^{0,1}}+\|X\|_{\mathcal{X}_T}^3\Big)
    \end{equation}

    Combining the conservation of the $L^2$ norm with \eqref{eq:H01}, \eqref{eq:H10} and \eqref{eq:Linfty} gives \eqref{ineq:boostrap-estimate} and the proof of Proposition \ref{prop:bootstrap-estimate} is complete.

\end{proof}

\section{Proof of Theorem \ref{th:principal}}\label{Section:proofofth}

All the elements are set to prove Theorem \ref{th:principal}. First we prove that the local solution of Theorem \ref{th:local-existence} is global, and then we show that it satisfies the scattering result \eqref{ineq:scattering}.

\begin{proposition}\label{prop:globalsol}
    The local solution $X$ provided by Theorem \ref{th:local-existence} is global in $C(\R,H^{1,0},L_\omega^2\cap H^{0,1},L_\omega^2)$ and verifies the following estimate: 

    \begin{equation}\label{ineq:decreasingestimate}
        \|X(t)\|_{L_x^\infty,L_\omega^2}\lesssim \varepsilon(1+ t)^{-1/2}.
    \end{equation}
\end{proposition}

\begin{proof} Let $T> 1$. We set $\epsilon>0$ sufficiently small such that we can apply Theorem \ref{th:local-existence} and $C_0$ the constant associated to $T$. 

    We can suppose that $C_0>1$ and we define $0<\varepsilon'<\varepsilon$ such that: $\varepsilon'+C_0^3C_1\varepsilon'^3\leq \frac{C+1}{2}\varepsilon'$, where we recall that $C_1$ is the constant independent of $T$ in Proposition \ref{prop:bootstrap-estimate}.

    Let $X$ be the local solution of \eqref{Cauchyprob} given by Theorem \ref{th:local-existence}, with $\|X_0\|_{H^{1,0},L_\omega^2\cap H^{0,1},L_\omega^2}\leq \varepsilon'$. 
    
    We show that $T^*:=\text{sup}\{T;\| X\|_{L_{T^*}^\infty,(H^{1,0},L_\omega^2\cap H^{0,1},L_\omega^2)}\leq C\varepsilon'\}=\infty$ by contradiction. Firstly, by Theorem \ref{th:local-existence} and by definition of $X$, $T^*>1$. Let us suppose by contradiction that $T^*<\infty$, then by Proposition \ref{prop:bootstrap-estimate} 

    $$\| X\|_{L_{T^*}^\infty,(H^{1,0},L_\omega^2\cap H^{0,1},L_\omega^2)}\leq \varepsilon'+C_1\| X\|_{\mathcal{X}_{T^*}}^3\leq  \varepsilon'+C^3C_1\varepsilon'^3\leq \frac{C+1}{2}\varepsilon'<C\varepsilon',$$

    which contradicts the definition of $T^*$, by continuity of the solution in $H^{1,0},L_\omega^2\cap H^{0,1},L_\omega^2$. Thus $T^*=\infty$ and the solution is global. 
\end{proof}

We also prove another lemma giving an estimate similar than \eqref{eq:estimateforR} but for the $L^2$ norm in space.

\begin{lemma}\label{lemma:estimateRL2}
    For any $X\in \mathcal{X}_T$, for $s\geq 1$, we have: 

    \begin{equation}\label{ineq:estimateRL2}
        \| R(s,\xi)\|_{H^{\theta,0}_\xi,L_\omega^2\cap H^{0,\theta}_\xi,L_\omega^2}\lesssim (s^{-1-\delta+3\alpha}+s^{-1-\delta'})\|X\|_{\mathcal{X}_T}^3,
    \end{equation}

    for any $0<\delta<\frac{1}{2}$ and for $\delta'<0$.
\end{lemma}

\begin{proof} \textbf{Estimation of the $H^{\theta,0}_\xi,L_\omega^2$ norm}
    Using the expressions \eqref{eq:FourF1} and \eqref{eq:FourF2} we have 

    \begin{equation}\label{eq:F_1L2}
        \mathcal{F}^{-1}_{\eta,\sigma}(F_1)(x,y)=2\pi e^{i\xi(x+y)}\mathcal{F}_{x'}\Big(\E[Z(x'-x)\overline{Z}(x')]Z(x'-y)\Big)(\xi),
    \end{equation}

    and

    \begin{equation}\label{eq:F_2L2}
        \mathcal{F}^{-1}_{\eta,\sigma}(F_2)(x,y)=2\pi e^{i\xi(x+y)}\mathcal{F}_{x'}\Big(\E[Z(x'-y)\overline{Z}(x')]Z(x'-x)\Big)(\xi).
    \end{equation}

    Then, by expression \eqref{eq:defR} we have$$\|R(s,\xi)\|_{H^{\theta,0}_\xi,L_\omega^2}\lesssim s^{-1-\delta}\big\|\langle \nabla\rangle^\theta\int dxdy\ \lvert x\rvert^\delta \lvert y \rvert^\delta \mathcal{F}_{x'}\Big(\E[Z(x'-x)\overline{Z}(x')]Z(x'-y)\Big)(\xi)\big\|_{L_\xi^2,L_\omega^2}.$$

    Thus by Plancherel's formula $$\|R(s,\xi)\|_{H^{\theta,0}_\xi,L_\omega^2}\lesssim s^{-1-\delta}\big\|\int dxdy\ \lvert x\rvert^\delta \lvert y \rvert^\delta \E[Z(x'-x)\langle x'\rangle^\theta\overline{Z}(x')]Z(x'-y)\big\|_{L_{x'}^2,L_\omega^2}.$$

    Using the fact that $$\lvert x \rvert^\delta\lvert y\rvert^\delta\lesssim \Big( \lvert x'-x\rvert^\delta+\lvert x' \rvert^\delta\Big)\Big( \lvert x'-y \rvert^\delta+\lvert x' \rvert^\delta\Big),$$

    we get $$\|R(s,\xi)\|_{H^{\theta,0}_\xi,L_\omega^2}\lesssim s^{-1-\delta} \|\langle x\rangle^{\delta}Z\|_{L_x^1,L_\omega^2}^2\|\langle x\rangle^{2\delta+\theta}Z\|_{L_x^2,L_\omega^2}. $$

    And if $0<\delta<\frac{1}{2}$ and $2\delta+\theta\leq 1$ we obtain

    \begin{equation}\label{eq:Hthet0}
        \|R(s,\xi)\|_{H^{\theta,0},L_\omega^2}\lesssim s^{-1-\delta+3\alpha}\|X\|_{\mathcal{X}_T}^3.
    \end{equation}

    \textbf{Estimation of the $H^{0,\theta}_\xi,L_\omega^2$ norm}

    We recall that by definition
    $$\Hat{C}(X,X,X)(t,\xi)=\int_1^t ds\ R(s,\xi). $$

    And by using the same method than in the proof of proposition \ref{prop:bootstrap-estimate} to obtain \eqref{eq:H011} and \eqref{eq:H012}, we claim that we have

    \begin{equation}\label{H0thinterp1}
        \|\langle \xi\rangle R(s,\xi)\|_{L_\xi^2,L_\omega^2}\lesssim s^{\alpha-1}\|X\|_{\mathcal{X}_T}^3.
    \end{equation}
    
    And, the estimation \eqref{eq:Hthet0} for $\theta=0$ and $\delta=\frac{1}{4}$ gives \begin{equation}\label{H0thinterp2}
        \|R(s,\xi)\|_{L_\xi^2,L_\omega^2}\lesssim s^{-1-\frac{1}{4}+3\alpha}\|X\|_{\mathcal{X}_T}^3.
    \end{equation}

    Interpolating \eqref{H0thinterp1} and \eqref{H0thinterp2} gives 

    \begin{equation}\label{H0interp}
        \|R(s,\xi)\|_{H^{0,\theta}_\xi,L_\omega^2}\lesssim s^{-1-(1-\theta)/4+\alpha(3-2\theta) }\|X\|_{\mathcal{X}_T}^3.
    \end{equation}

    Combining \eqref{eq:Hthet0} and \eqref{H0interp} gives the result, with $\delta'=(1-\theta)/4-\alpha(3-2\theta)>0$ by definition of $\alpha$.
\end{proof}

We now can conclude with the proof of Theorem \ref{th:principal}.

\begin{proof}[Proof of Theorem \ref{th:principal}]
    Equality \eqref{eq:duhamelZ} gives that for any $t>s>1$: $$\|\hat{Z}(t,\xi)-\hat{Z}(s,\xi)\|_{H^{\theta,0}_\xi,L_\omega^2\cap H^{0,\theta}_\xi,L_\omega^2\cap L_\xi^\infty,L_\omega^2 }\leq \int_s^t \|R(\tau,\xi)\|_{H^{\theta,0}_\xi,L_\omega^2\cap H^{0,\theta}_\xi,L_\omega^2\cap L_\xi^\infty,L_\omega^2}d\tau.$$

    Then using Lemma \ref{lemma:estimateR} and Lemma \ref{lemma:estimateRL2} gives the existence of $\delta> 0$ such that 

    \begin{equation}\label{ineq:Cauchy}
        \|\hat{Z}(t,\xi)-\hat{Z}(s,\xi)\|_{H^{\theta,0}_\xi,L_\omega^2\cap H^{0,\theta}_\xi,L_\omega^2\cap L_\xi^\infty,L_\omega^2}\lesssim t^{-\delta}-s^{-\delta}.
    \end{equation}

    Thus there exists $W\in H^{\theta,0}_\xi,L_\omega^2\cap H^{0,\theta}_\xi,L_\omega^2\cap L_\xi^\infty,L_\omega^2$ such that : $$W:=\underset{t\to \infty}{\text{lim}}\hat{Z}(t) \text{ in } L_\xi^2,L_\omega^2\cap L_\xi^\infty,L_\omega^2. $$

    To conclude we let $s\to \infty$ in \eqref{ineq:Cauchy}, which gives \eqref{ineq:scattering}, and the proof of Theorem \ref{th:principal} is complete. 
\end{proof}

\section{Proof of scattering for the density operators framework}\label{Sec:proofthop}

In this part we give the proof of Theorem \ref{Cor:scatteringoper}, making the link between the density operators and random fields. 

\begin{proof}[Proof of Theorem \ref{Cor:scatteringoper}]

$\gamma_0$ is of Trace class, so there exists a family of non-negative real $(\alpha_n)_{n\in\N}$ and an orthonormal family of $L^2(\R)$ $(e_n)_{n\in\N}$ such that
$$\gamma_0=\sum_{n\in\N}\alpha_n|e_n\rangle\langle e_n|.$$

We set $(g_n)_{n\in\N}$ a family of complex centered normalised independent Gaussian variable, and we define $$X_0=\sum_{n\in\N}\sqrt{\alpha_n}g_ne_n.$$

We note that $\E[\overline{g_n}g_m]=0$ if $n\neq m$ and $1$ otherwise. 

Then, for $v\in L^2(\R)$ $$\E[\langle v,X_0\rangle X_0]=\sum_{n,m} \sqrt{\alpha_n\alpha_m}\langle v,e_n\rangle e_m \E[\overline{g_n}g_m]= \sum_{n} \alpha_n\langle v,e_n\rangle e_n=\gamma_0(v),$$

thus: $\E[|X_0\rangle \langle X_0|]=\gamma_0$.

Moreover, we have that: $$\|X_0\|_{L^2(\Omega,H^{0,1}\cap H^{1,0}}^2=\E[\langle X_0,\langle \nabla \rangle X_0\rangle]+\E[\langle X_0,\langle x \rangle X_0\rangle],$$

and because $Tr(AB)=Tr(BA)$
$$\|X_0\|_{L^2(\Omega,H^{0,1}\cap H^{1,0}}^2=\E[\text{Tr}(|X_0\rangle\langle X_0|\langle \nabla \rangle)]+\E[\text{Tr}(|X_0\rangle\langle X_0|\langle x\rangle)],$$

then: $$\|X_0\|_{L^2(\Omega,H^{0,1}\cap H^{1,0}}^2=\text{Tr}(\E[|X_0\rangle\langle X_0|]\langle \nabla \rangle)+\text{Tr}(\E[|X_0\rangle\langle X_0|]\langle x\rangle)=\text{Tr}(\langle \nabla \rangle \gamma_0)+\text{Tr}(\langle x \rangle \gamma_0)\leq \varepsilon.$$

Then applying Theorem \ref{th:principal}, we obtain a solution $X\in \mathcal C(\R,H^{1,0},L_\omega^2\cap H^{0,1},L_\omega^2)$ of \eqref{Cauchyprob} with initial condition $X(t=1)=e^{i\Delta}X_0$ and $W\in L_\xi^2,L_\omega^2$ such that for $t\geq 1$

\begin{equation}
    \|\hat{X}-e^{-it\xi^2}\Hat W\|_{L_\xi^2,L_\omega^2}\lesssim t^{-\delta}.
\end{equation}

We now claim that the solution for \eqref{Cauchyprob-density-matrices} we search is \begin{equation}
    \gamma:=\E[|X\rangle\langle X|]:=u\mapsto \big(x\mapsto \E(X(x)\langle X,u\rangle_{L_x^2})\big).
\end{equation} 

$\gamma$ is a solution of \eqref{Cauchyprob-density-matrices} with initial condition $\gamma_0$, by definition of $X$ and $X_0$. We now show that it scatters at infinity, and we introduce the operator $$\gamma_\infty:=\E[|W\rangle\langle W|].$$

First we prove that the application $(Z_1,Z_2)\mapsto \E[|Z_1\rangle \langle Z_2|]$ is continuous from $(L_{x,\omega}^2)^2$ to $\mathfrak S^1$.

We have by Fubini and because the operator $\E[|Z_1\rangle \langle Z_2|]$ is of integral kernel $\E[Z_1(x)\overline{Z_2}(y)]$ 

$$\big\|\E[|Z_1\rangle \langle Z_2|]\big\|_{\mathfrak S^1}=\int_{\R} \big|\E[Z_1(x)\overline{Z_2}(x)]\big|dx.$$

So, by Cauchy-Schwarz inequality

$$\big\|\E[|Z_1\rangle \langle Z_2|]\big\|_{\mathfrak S^1}\leq \int_{\R} \|Z_1(x)\|_{L_\omega^2}\|Z_2(x)\|_{L_\omega^2}dx\leq \|Z_1\|_{L_{x,\omega}^2}\|Z_2\|_{L_{x,\omega}^2} .$$

We then compute that

$$\gamma-\gamma_\infty=\E[|S(t)X(t)-W\rangle\langle S(t)X(t)|]+\E[|W\rangle \langle W-S(t)X(t)|].$$

Combining the continuity of the application $(Z_1,Z_2)\mapsto \E[|Z_1\rangle \langle Z_2|]$ from $(L_{x,\omega}^2)^2$ to $\mathfrak S^2$, with the $L^2$ part of the inequality \eqref{ineq:scattering} and the fact that $X\in L_t^\infty,L_{x,\omega}^2$ gives the scattering result \eqref{ineq:scatteringopdens}.

For the result of global well-posedness, we mention the work \cite{ASdesuzzonglobwp} (section 6). We can use a similar method to prove that, on $\Sigma$ the set of non negative operators $\gamma$ such that $\text{Tr}(\langle \nabla \rangle \gamma)+\text{Tr}(\langle x \rangle \gamma)$ is finite, we have a unique solution of \eqref{Cauchyprob-density-matrices} with initial condition $\gamma_0$. Indeed, we can introduce a distance $d$ in this space defined by: $$d(\gamma_1,\gamma_2)=\underset{X_i\sim \gamma_i}{inf}\|X_1-X_2\|_{L^2(\Omega,H^{1,0}\cap H^{0,1})},$$where $\sim$ stands for "is a Gaussian random field of covariance". And if we have $\gamma$ a solution of \eqref{Cauchyprob-density-matrices}, we can find a solution $X\in \mathcal{C}(\R,L^2(\Omega,H^{1,0}\cap H^{0,1})$ of \eqref{Cauchyprob}, and of covariance $\gamma$. Then, we obtain the unicity of the solution of \eqref{Cauchyprob-density-matrices} by unicity of the solution of \eqref{Cauchyprob}, because if $X_i\sim \gamma_i$ then we have: 
$d(\gamma_1,\gamma_2)\leq \|X_1-X_2\|_{L^2(\Omega,H^{1,0}\cap H^{0,1})}.$
    
\end{proof}

\newpage

\appendix 

\section{Sketch of proof of Theorem \ref{th:local-existence}}\label{appendix:locex}

Solving the Cauchy problem \eqref{Cauchyprob} with initial condition $X_0$, is equivalent to solving the fixed point problem 

\begin{equation}
    X(t,x)=X_0+e^{it\Delta}\Tilde{C}(X,X,X),
\end{equation}

with $$\Tilde{C}(X,X,X)=-i\int_0^te^{-is\Delta}\Big( \big(w*\E[\lvert X(s)\rvert^2]\big)X(s)-\int w(x-y)\E[\overline{X(s,y)}X(s,x)]X(s,y)dy \Big)ds.$$

We just estimate the first term in $\Tilde{C}$, the other can be dealt the same way. We first estimate $\partial_x w*\E[\lvert X(s)\rvert^2]\big) X(s)$ in $L_2$. For this we compute, by similar methods as in Section \ref{Section:proofofprop}, the following (the other terms can be treated the same way) 

$$\mathcal{F}\big(w*\E[\lvert X(s)\rvert^2]\big)\partial_x X(s)\big)(\xi)=\int d\eta d\sigma\ \hat{w}(\eta)\E\big[\overline{\hat{X}(s,\sigma-\eta)}\hat{X}(s,\sigma)\big](\xi-\eta)\hat{X}(s,\xi-\eta).$$

We get, integrating with respect to $\xi$, then to $\eta$ and finally to $\sigma$ we get

$$\|\big(w*\E[\lvert X(s)\rvert^2]\big)\partial_x X(s)\big)\|_{L_x^2,L_\omega^2}\lesssim \|\hat{X}\|_{L_\xi^1,L_\omega^2}^2\|\xi\hat{X}\|_{L_\xi^2,L_\omega^2}^2.$$

Then we obtain, by Hölder inequality \begin{equation}
    \|\partial_x \big(w*\E[\lvert X(s)\rvert^2]\big)X(s)\big)\|_{L_x^2,L_\omega^2}\lesssim \|X\|_{H_x^{1,0},L_\omega^2}^3.
\end{equation}

We can conclude that 

 \begin{equation}
    \|\Tilde{C}\|_{H^{1,0}_x,L_\omega^2}\lesssim t \|X\|_{H_x^{1,0},L_\omega^2}^3.
\end{equation}

Now we estimate $\|x\Tilde{C}\|_{L_x^2,L_\omega^2}=\|\partial_\xi\mathcal{F}(\Tilde{C})\|_{L_\xi^2,L_\omega^2}$. Recalling formulas \eqref{eq:I_1}, \eqref{eq:I_2} and \eqref{eq:C} we just estimate $\partial_\xi I_1'$, where

$$I_1'=\frac{-i}{\sqrt{2\pi}}\int_0^tds \int d\eta d\sigma\ e^{2is\eta (\xi-\sigma)}\hat{w}(\eta)\E[\overline{\hat{Z}(s,\sigma-\eta)}\hat{Z}(\sigma)]\hat{Z}(\xi-\eta).$$

We have \begin{equation}
    \begin{array}{rcl}
         \partial_\xi I_1' & = & \frac{-i}{\sqrt{2\pi}}\int_0^tds \int d\eta d\sigma\ e^{2is\eta (\xi-\sigma)}\hat{w}(\eta)\E[ \overline{\hat{Z}(s,\sigma-\eta)}\hat{Z}(\sigma)]\partial_\xi\hat{Z}(\xi-\eta)  \\
         & & + 2i\frac{-i}{\sqrt{2\pi}}\int_0^tds \ s \int d\eta d\sigma\ \eta e^{2is\eta (\xi-\sigma)}\hat{w}(\eta)\E[ \overline{\hat{Z}(s,\sigma-\eta)}\hat{Z}(\sigma)]\hat{Z}(\xi-\eta)
    \end{array}
\end{equation}

For the first term we integrate over $\xi$ then over $\eta$ and then over $\sigma$ to obtain

$$\|\int d\eta d\sigma\ e^{2is\eta (\xi-\sigma)}\hat{w}(\eta)\E[ \overline{\hat{Z}(s,\sigma-\eta)}\hat{Z}(\sigma)]\partial_\xi\hat{Z}(\xi-\eta)\|_{L_\xi^2,L_\omega^2}\lesssim \|\hat{Z}\|_{L_\xi^1,L_\omega^2}^2\|\partial_\xi \hat{Z}\|_{L_\xi^2,L_\omega^2}. $$

Then by Hölder inequality 

$$\|\int d\eta d\sigma\ e^{2is\eta (\xi-\sigma)}\hat{w}(\eta)\E[ \overline{\hat{Z}(s,\sigma-\eta)}\hat{Z}(\sigma)]\partial_\xi\hat{Z}(\xi-\eta)\|_{L_\xi^2,L_\omega^2}\lesssim \| X\|_{H_x^{1,0},L_\omega^2}^2\| Z\|_{H_x^{0,1},L_\omega^2}$$

For the second term, we denote by $\mathcal{A}= \int d\eta d\sigma\ \eta e^{2is\eta (\xi-\sigma)}\hat{w}(\eta)\E[ \overline{\hat{Z}(s,\sigma-\eta)}\hat{Z}(\sigma)]\hat{Z}(\xi-\eta)$, we have

$$\begin{array}{rcl}
     \mathcal{A} & = & \int d\eta d\sigma\ e^{2is\eta (\xi-\sigma)}\hat{w}(\eta)\E[ \overline{\hat{Z}(s,\sigma-\eta)}\sigma\hat{Z}(\sigma)]\hat{Z}(\xi-\eta)\\
     & & -\int d\eta d\sigma\  e^{2is\eta (\xi-\sigma)}\hat{w}(\eta)\E[(\sigma-\eta) \overline{\hat{Z}(s,\sigma-\eta)}\hat{Z}(\sigma)]\hat{Z}(\xi-\eta)
\end{array}$$

By a change of variable we get that there is $C_1$ and $C_2$ depending on $\eta$, $\sigma$ and $\xi$ such that 

$$\begin{array}{rcl}
     \mathcal{A} & = & \int d\eta d\sigma\ e^{2isC_1}\hat{w}(\eta-\xi)\E[ \overline{\hat{Z}(s,\eta-\sigma)}(\xi-\sigma)\hat{Z}(\xi-\sigma)]\hat{Z}(\eta)\\
     & & -\int d\eta d\sigma\  e^{2isC_2}\hat{w}(\eta-\xi)\E[(\eta-\xi) \overline{\hat{Z}(s,\eta-\xi)}\hat{Z}(\sigma-\eta)]\hat{Z}(\eta)
\end{array},$$

thus reasoning as before gives that 

$$\| \partial_\xi I_1'\|_{L_\xi^2,L_\omega^2}\lesssim  t\| X\|_{H_x^{1,0},L_\omega^2}^2\| Z\|_{H_x^{0,1},L_\omega^2}+t^2\|X\|_{H_x^{1,0},L_\omega^2}^3.$$

So we have the estimate for $t\geq 1$

$$\|\Tilde{C}\|_{H_x^{0,1}\cap H_x^{1,0},L_\omega^2}\lesssim t\| X\|_{H_x^{1,0},L_\omega^2}^2\| Z\|_{H_x^{0,1},L_\omega^2}+t^2\|X\|_{H_x^{1,0},L_\omega^2}^3.$$

And using classical arguments to solve the fixed point problem in 

$$ \mathcal{X}'_T= \Big\{X : \|X\|_{\mathcal{X}'_T}:=\|X\|_{L_T^\infty,\Dot{H}_x^{1,0},L_\omega^2}+\|Z\|_{L_T^\infty,\Dot{H}_x^{0,1},L_\omega^2}+\| X\|_{L_T^\infty,L_x^2,L_\omega^2}<\infty \Big\}, $$ 

we claim that for $T\geq1$ there is a unique solution in $L^\infty\big([0,T],B_{\mathcal{X}'_T}(0,C\varepsilon)\big)$ for some constant $C>0$ and $\varepsilon>0$ small enough

To conclude that the solution is also in $L_T^\infty,L_x^\infty,L_\omega^2$, we just notice that the estimate we gave on $R(s,\xi)$ in order to control the $L_T^\infty,L_x^\infty,L_\omega^2$ norm, only depends on the $H_x^{0,1},L_\omega^2$ norm of $Z$.

\newpage

\bibliographystyle{amsplain}
\bibliography{biblio}

\end{document}